\documentclass[11pt]{article}
\usepackage{amsmath,amssymb,amsthm,amsfonts,amstext,amsbsy,amscd,mathrsfs,graphics}

\newtheorem{lemma}{Lemma}
\newtheorem{thm}{Theorem}

\newtheorem{prop}{Proposition}
\newtheorem{corl}{Corollary}

\usepackage{amsmath}
\usepackage{amssymb}
\usepackage{graphicx}

\usepackage{color}

\newcommand{\N}{\mathbb N}
\newcommand{\R}{\mathbb R}

\begin{document}

\title{On the Cs\'aki-Vincze transformation}
\date{\today}
\maketitle
\begin{center}
\renewcommand{\thefootnote}{(\arabic{footnote})}
  \scshape Hatem Hajri\footnote{Email: hatemfn@yahoo.fr}
\renewcommand{\thefootnote}{\arabic{footnote}}\setcounter{footnote}{0}
\end{center}
\hglue0.02\linewidth\begin{minipage}{0.9\linewidth}
\begin{center}
{Universit\'e du Luxembourg.}
\end{center}
\end{minipage}
\maketitle
\begin{abstract}
Cs\'aki and Vincze have defined in 1961 a discrete transformation $\mathcal T$ which applies to simple random walks and is measure preserving. In this paper, we are interested in ergodic and assymptotic properties of $\mathcal T$. We prove that $\mathcal T$ is exact : $\bigcap_{k\geq 1} \sigma(\mathcal T^k(S))$ is trivial for each simple random walk $S$ and give a precise description of the lost information at each step $k$. We then show that, in a suitable scaling limit, all iterations of $\mathcal T$ ''converge'' to the corresponding iterations of the continous L\'evy transform of Brownian motion. Some consequences are also derived from these two results.
%The L\'evy transform of a Brownian motion $B$ is given by $T(B)_t=\int_{0}^{t}\textrm{sgn}(B_s)dB_s$.  We then consider a suitable simple random walk $S^n$ embedded in $B$ depending on $n$ and prove that for each $k$, $\mathcal T^k(S^n)$ suitably normalized and time scaled converges to $T^k(B)$ in probability in $\mathbb C(\R_+,\R)$ as $n\rightarrow\infty$. 
\end{abstract}
\section{Introduction and main results.}
Let $B$ be a Brownian motion, then $T(B)_t=\int_{0}^{t}\textrm{sgn}(B_s)dB_s$ is a Brownian motion too. Iterating $T$ yields a family of Brownian motions $(B^n)_n$ given by
$$B^0=B,\ \ B^{n+1}=T(B^n).$$
We call $B^n$ the $n$-iterated L\'evy transform of $B$. At least two transformations of simple random walks have
been studied in the literature as discrete analogues to $T$. For a simple random walk (SRW) $S$, Dubins and Smorodinsky \cite{MR1231991} define the L\'evy transform $\Gamma(S)$ of $S$ as the SRW obtained by skipping plat paths from 
$$n\longmapsto |S_n|-L_n$$
\noindent where $L$ is a discrete analogous of local time. Their fundamental result says that $S$ can be recovered from the signs of the excursions of $S, \Gamma(S), \Gamma^2(S),\cdots$ and a fortiori $\Gamma$ is ergodic. Later, another discrete L\'evy transformation $F$ was given by Fujita \cite{MR2417970}:
$$F(S)_{k+1}-F(S)_k=\textrm{sgn}(S_k)(S_{k+1}-S_k),\ \textrm{with the convention}\ \textrm{sgn}(0)=-1.$$
However, $F$ is not ergodic by the main result of \cite{MR2417970}. Our main purpose in this paper is to study a transformation already obtained by Cs\'aki and Vincze.\\
Let $\mathbb W=\mathbb C(\R_+,\R)$ be the Wiener space equipped with the distance  
$${d}_U(w,w')=\sum_{n\geq 1} 2^{-n}\big(\sup_{0\leq t\leq n}|w(t)-w'(t)|\wedge1\big).$$
We endow $\mathbb E=\mathbb W^{\N}$ with the product metric defined for each $x=(x_k)_{k\geq 0}, y=(y_k)_{k\geq 0}$ by
$${d}(x,y)=\sum_{k\geq 0} 2^{-k}(d_U(x_k,y_k)\wedge1).$$
Thus $(\mathbb E,{d})$ is a separable complete metric space.\\
For each SRW $S$ and $h\geq 0$, we denote by $\mathcal T^h(S)$ the $h$-iterated Cs\'aki-Vincze transformation (to be defined in Section \ref{hop}) of $S$ with the convention $\mathcal T^0(S)=S$.\\
Let $B$ be a Brownian motion defined on $(\Omega,\mathcal A,\mathbb P)$. For each $n\geq 1$, define $T^n_0=0$ and for all $k\geq 0$, 
$$T^n_{k+1}=\inf\bigg\{t\geq T^n_k : |B_t-B_{T^n_k}|=\frac{1}{\sqrt{n}}\bigg\}.$$
Then $S^n_k=\sqrt{n} B_{T^n_k},\ k\geq 0,$ is a SRW and we have the following
\begin{thm}\label{g}
\begin{itemize}
\item[(i)] For each SRW $S$ and $h\geq 0$, $\mathcal T^h(S)$ is independent of $(S_j, j\leq h)$ and a fortiori $$\bigcap_{h\geq 0} \sigma(\mathcal T^h(S))$$ is trivial.
\item[(ii)] For each $n\geq 1$, $h\geq 0$ and $t\geq 0$, define
$$S^{n,h}(t)=\frac{1}{\sqrt{n}}\mathcal T^h(S^n)_{\lfloor nt \rfloor}+\frac{(nt-\lfloor nt \rfloor)}{\sqrt{n}}\big(\mathcal T^h(S^n)_{\lfloor nt \rfloor+1}-\mathcal T^h(S^n)_{\lfloor nt \rfloor}\big).$$
Then 
$$(S^{n,0},S^{n,1},S^{n,2},\cdots)$$
converges to 

$$(B^0,B^1,B^2,\cdots)$$
in probability in $\mathbb E$ as $n\rightarrow\infty$. 
\end{itemize}
\end{thm}
Theorem \ref{g} (i) says that $\mathcal T$ is exact, but there are more informations in the proof. For instance, the random vectors $(S_1,S_2,\cdots,S_n)$ and $(S_1,\mathcal T(S)_1,\cdots,\mathcal T^{n-1}(S)_1)$ generate the same $\sigma$-field; so the whole path $(S_n)_{n\geq 1}$ can be encoded in the sequence $(\mathcal T^{n}(S)_1)_{n\geq 0}$ which is stronger than exactness.
From Theorem \ref{g} (i), we can deduce the following
\begin{corl}\label{yeux}

Fix $p\geq 2$ and let $\alpha^i=(\alpha^i_n)_{n\geq 1}$, $i\in[1,p]$ be $p$ nonegative sequences such that $$\alpha^1_n\longrightarrow+\infty,\ \ \alpha^i_n-\alpha^{i-1}_n\longrightarrow+\infty\ \textrm{as}\ n\rightarrow\infty\ \textrm{for all}\ i\in[2,p].$$
Let $S$ be a SRW and $X_0,X_1,\cdots,X_{p}$ be $p+1$ independent Brownian motions. For $n\geq 1, h\geq 0, t\in\R_+$, define 
\begin{equation}\label{moto}
S^{h}_n(t)=\frac{1}{\sqrt{n}}\mathcal T^h(S)_{\lfloor nt \rfloor}+\frac{(nt-\lfloor nt \rfloor)}{\sqrt{n}}\big(\mathcal T^h(S)_{\lfloor nt \rfloor+1}-\mathcal T^h(S)_{\lfloor nt \rfloor}\big).
\end{equation}
Then $$\bigg(S^0_n,S^{\lfloor n\alpha^1_n\rfloor}_n,\cdots,S^{\lfloor n\alpha^p_n\rfloor}_n\bigg)\xrightarrow[\text{$n\rightarrow +\infty $}]{\text{law}}\big(X_0,X_1,\cdots,X_{p}\big)\ \textrm{in}\ \mathbb W^{p+1}.$$

\end{corl}
A natural question which is actually motivated by the famous question of ergodicity of the L\'evy transformation $T$ as it will be discussed in Section \ref{ra}, is to focus on sequences $(h_n)_n$ tending to $\infty$ and satisfying  
\begin{equation}\label{em}
\lim_{n\rightarrow\infty}\bigg(S^{n,{h_n}}(t)-B^{{h_n}}_t\bigg)=0\ \textrm{in probability}.
\end{equation}
Such sequences exist and when (\ref{em}) holds, we necessarily have $\lim_{n\rightarrow\infty}\frac{h_n}{n}=0$. This is summarized in the following
\begin{prop}\label{feda}
With the same notations of Theorem \ref{g}:
\begin{itemize}
 \item[(i)] There exists a family $(\alpha^i)_{i\in\N}$ of nondecreasing sequences $\alpha^i=(\alpha^i_n)_{n\in\N}$ with values in $\N$ such that 
$$\alpha^0_n\longrightarrow+\infty,\ \ \alpha^i_n-\alpha^{i-1}_n\longrightarrow+\infty\ \textrm{as}\ n\rightarrow\infty\ \textrm{for all}\ i\geq 1$$
and moreover 
$$\lim_{n\rightarrow\infty}\bigg(S^{n,{\alpha^i_n}}-B^{\alpha^i_n}\bigg)=0\ \textrm{in probability in}\ \mathbb W$$
for all $i\in\N$.
\item [(ii)] If $(h_n)_n$ is any integer-valued sequence such that $\frac{h_n}{n}$ does not tend to $0$, then there exists no $t>0$ such that (\ref{em}) holds.
\end{itemize}
\end{prop}
 
In the next section, we review the Cs\'aki-Vincze transformation, establish part (i) of Theorem \ref{g} and show that $(S,\mathcal T(S),\cdots,\mathcal T^h(S),\cdots)$ ''converges'' in law to $(B,T(B),\cdots,T^h(B),\cdots)$. To prove part (ii) of Theorem \ref{g}, we use the simple idea : if $Z_n$ converges in law to a constant $c$, then the convergence holds also in probability. The other proofs are based on the crucial property of the transformation $\mathcal T$ : $\mathcal T^h(S)$ is independent of $\sigma(S_j, j\leq h)$ for each $h$. In Section \ref{ra}, we compare our work with \cite{MR1231991} and \cite{MR2417970} and discuss the famous question of ergodicity of $T$.
\section{Proofs.}
\subsection{The Cs\'aki-Vincze transformation and convergence in law.}\label{hop}
For the sequel, we recommand the lecture of the pages 109 and 110 in \cite{MR2168855} (Theorem \ref{khali} below). Some consequences (see Proposition \ref{wop} below) have been drawn in \cite{MR50101010} (Sections 2.1 and 2.2). We also notice that our stating of this result is slightly different from \cite{MR2168855} and leave to the reader to make the obvious analogy.
\begin{thm}\label{khali}(\cite{MR2168855} page 109)
Let $S = (S_n)_{n\geq 0}$ be a SRW defined on $(\Omega,\mathcal A,\mathbb P)$ and $X_i=S_i-S_{i-1}, i\geq1$. Define $\tau_{0}=0$ and for $l\geq 0$,
$$\tau_{l+1}=\min\big{\{i>\tau_{l} : S_{i-1}S_{i+1}< 0\big\}}.$$
Set $$\overline{X}_j=\sum_{l\geq0}(-1)^{l}X_1X_{j+1}1_{\{\tau_{l}+1\leq j\leq\tau_{l+1}\}}.$$
Then $\overline{S}_0=0,\ \overline{S}_n = \overline X_1+\cdots+\overline X_n, n\geq 1$ is a SRW. Moreover if $Y_n:=\overline S_n-\displaystyle{\min_{k\leq n}} \overline S_k$, then for all $n\in\N$, 
\begin{equation}\label{as}
|Y_n -|S_n|| \leq 2.
\end{equation}
We call $\overline S=\mathcal T(S)$, the Cs\'aki-Vincze transformation of $S$ (see the figures $1$ and $2$ below).
\end{thm}
\begin{figure}[h]
\begin{center}
\resizebox{12.5cm}{7cm}{\input{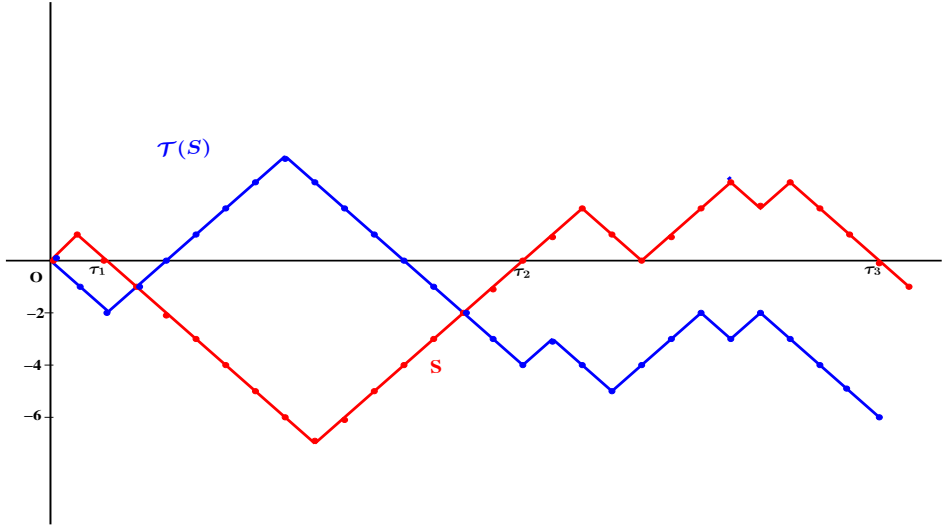}}
\caption{$S$ and $\mathcal T(S)$.}
\end{center}
\end{figure}
\begin{figure}[h]
\begin{center}
\resizebox{12.5cm}{5.3cm}{\input{fig2_article4.pstex_t}}
\caption{$|S|$ and $Y$.}
\end{center}
\end{figure}
\noindent Note that $(-1)^l X_1$ is simply equal to $\textrm{sgn}(S)_{|[\tau_l+1,\tau_{l+1}]}(:=X_{\tau_l+1})$ which can easily be checked by induction on $l$. Thus for all $j\in[\tau_l+1,\tau_{l+1}]$,
$$\overline{X}_j=\textrm{sgn}(S)_{|[\tau_l+1,\tau_{l+1}]}(S_{j+1}-S_j)$$
or equivalently 
\begin{equation}\label{monz}
\mathcal T(S)_j-\mathcal T(S)_{j-1}=\textrm{sgn}(S_{j-\frac{1}{2}})(S_{j+1}-S_j)
\end{equation}
where $t\longrightarrow S_t$ is the linear interpolation of $(S_n)_{n\geq 0}$. Hence, one can expect that $(S,\mathcal T(S)$ will ''converge'' to $(B,B^1)$ in a suitable sense. The following proposition has been established in \cite{MR50101010}. We give its proof for completeness.
\begin{prop}\label{wop}
With the same notations of Theorem \ref{khali}, we have 
\begin{enumerate}
 \item [(i)] For all $n\geq0$, $\sigma(\mathcal T(S)_j, j\leq n)\vee \sigma(S_1)=\sigma(S_j, j\leq n+1)$.
 \item [(ii)] $S_1$ is independent of $\sigma(\mathcal T(S))$.
\end{enumerate}
\end{prop}
\begin{proof}
(i) The inclusion $\subset$ is clear from (\ref{monz}). Now, for all $1\leq j\leq n$, we have $X_{j+1}=\displaystyle\sum_{l\geq0}(-1)^{l}X_1\overline X_j1_{\{\tau_{l}+1\leq j\leq\tau_{l+1}\}}$. As a consequence of $(iii)$ and $(iv)$ \cite{MR2168855} (page 110), for all $l\geq 0$,
$$\tau_{l}=\min{\{n\geq 0 , \mathcal T(S)_n=-2l\}}. $$
Thus $\tau_l$ is a stopping time with respect to the natural filtration of $\mathcal T(S)$ and as a result $\{\tau_{l}+1\leq j\leq\tau_{l+1}\}\in\sigma(\mathcal T(S)_h, h\leq j-1)$ which proves the inclusion $\supset$.\\
(ii) We may write for all $l\geq 1$,
$$\tau_{l}=\min{\{i>\tau_{l-1} : X_1S_{i-1}X_1S_{i+1}< 0\}}.$$
This shows that $\mathcal T(S)$ is $\sigma(X_1X_{j+1}, j\geq0)$-measurable and (ii) is proved.\\
\end{proof}
\noindent Note that
$$\mathcal T(S)=\mathcal T(-S),\ \ \sigma(\mathcal T^{h+1}(S))\subset\sigma(\mathcal T^h(S)),$$
which is the analogous of 
$$T(B)=T(-B),\ \ \sigma(T^{h+1}(B))\subset\sigma(T^h(B)).$$
The previous proposition yields the following
\begin{corl}\label{ko}
For all $n\geq 0$,
\begin{enumerate}
 \item [(i)] $\sigma(S)=\sigma(\mathcal T^n(S))\vee \sigma(S_k, k\leq n)$.
 \item [(ii)] $\sigma(\mathcal T^n(S))$ and $\sigma(S_k, k\leq n)$ are independent.
 \item [(iii)] The $\sigma$-field
 $$\mathcal G^{\infty}=\bigcap_{n\geq 0} \sigma(\mathcal T^n(S))$$
 is $\mathbb P$-trivial.
\end{enumerate}
\end{corl}
\begin{proof}
Set $X_i=S_i-S_{i-1}, i\geq 1$.\\
(i) We apply successively Proposition \ref{wop} (i) so that for all $n\geq 1$,
\begin{eqnarray}
\sigma(S)&=&\sigma(\mathcal T(S))\vee \sigma(S_1)\nonumber\\
&=&\sigma(\mathcal T^2(S))\vee\sigma(\mathcal T(S)_1)\vee \sigma(S_1)\nonumber\\
&=&\cdots\nonumber\\
&=&\sigma(\mathcal T^{n}(S))\vee \sigma(\mathcal T^{n-1}(S)_1)\vee\cdots\vee\sigma(\mathcal T(S)_1)\vee \sigma(S_1).\nonumber\
\end{eqnarray} 
To deduce (i), it suffices to prove that
\begin{equation}\label{pl}
\sigma(S_k, k\leq n)=\sigma(\mathcal T^{n-1}(S)_1)\vee\cdots\vee\sigma(\mathcal T(S)_1)\vee \sigma(S_1).
\end{equation}
Again Proposition \ref{wop} (i), yields
\begin{eqnarray}
\sigma(S_k, k\leq n)&=&\sigma(\mathcal T(S)_j, j\leq n-1)\vee \sigma(S_1)\nonumber\\
&=&\sigma(\mathcal T^2(S)_j, j\leq n-2)\vee\sigma(\mathcal T(S)_1)\vee \sigma(S_1)\nonumber\\
&=&\cdots\nonumber\\
&=&\sigma(\mathcal T^{n-1}(S)_1)\vee \sigma(\mathcal T^{n-2}(S)_1)\cdots\vee\sigma(\mathcal T(S)_1)\vee \sigma(S_1)\nonumber\
\end{eqnarray} 
which proves (\ref{pl}) and allows to deduce (i).\\
(ii) will be proved by induction on $n$. For $n=0$, this is clear. Suppose the result holds for $n$, then $S_1,\mathcal T^1(S)_1,\cdots, \mathcal T^{n-1}(S)_1, \mathcal T^n(S)$ are independent (recall (\ref{pl})). Let prove that $S_1,\mathcal T^1(S)_1,\cdots, \mathcal T^n(S)_1,\mathcal T^{n+1}(S)$ are independent which will imply (ii) by (\ref{pl}). Note that $\mathcal T^n(S)_1$ and $\mathcal T^{n+1}(S)$ are $\sigma(\mathcal T^n(S))$-measurable. By the induction hypothesis, this shows  that $(S_1, \mathcal T^1(S)_1,\cdots,\mathcal T^{n-1}(S)_1)$ and $(\mathcal T^n(S)_1, \mathcal T^{n+1}(S))$ are independent. But $\mathcal T^n(S)_1$ and $\mathcal T^{n+1}(S)$ are also independent by Proposition \ref{wop} (ii). Hence (ii) holds for $n+1$ and thus for all $n$.\\
(iii) Let $A\in\mathcal G^{\infty}$ and fix $n\geq 1$. Then $A\in\sigma(\mathcal T^n(S))$ and we deduce from (ii) that $A$ is independent of $\sigma(S_k, k\leq n)$. Since this holds for all $n$, $A$ is independent of $\sigma(S)$. As $\mathcal G^{\infty}\subset\sigma(S)$, $A$ is therefore independent of itself.
\end{proof}
Let $S$ be a SRW defined on $(\Omega,\mathcal A,\mathbb P)$ and recall the definition of $S^{h}_n(t)$ from (\ref{moto}). On $\mathbb E$, define 
$$Z^n(t_0,t_1,\cdots,t_h,\cdots)=\left(S_n^0(t_0), S_n^1(t_1), \cdots, S_n^h(t_h), \cdots\right)$$
and let $\mathbb P_n$ be the law of $Z^n$.
\begin{lemma}\label{gh}
The family $\{\mathbb P_n, n\geq 1\}$ is tight on $\mathbb E$.
\end{lemma}
\begin{proof} By Donsker theorem for each $h$, $S_n^h$ converges in law to standard Brownian motion as $n\rightarrow\infty$. Thus the law of each coordinate of $Z^n$ is tight on $\mathbb W$ which is sufficient to get the result (see \cite{MR838085} page 107). 
\end{proof}
\noindent\textbf{The limit process.} Fix a sequence $(m_n,n\in\N)$ such that $Z^{m_n}\xrightarrow[\text{$n\rightarrow +\infty $}]{\text{law}}Z$ in $\mathbb E$ where 
$$Z=\bigg(B^{(0)},B^{(1)},\cdots,B^{(h)},\cdots\bigg)$$
is the limit process. Note that $B^{(0)}$ is a Brownian motion. From (\ref{as}), we have $\forall n\geq 1, t\geq 0$
$$\left||S^0_n(t)|-(S_n^1(t)-\displaystyle{\min_{0\leq u\leq t}} S_n^1(u))\right|\leq \frac{2012}{\sqrt{n}}.$$
Letting $n\rightarrow\infty$, we get
$$|B^{(0)}_t|=B_t^{(1)}-\min_{0\leq u\leq t}B_u^{(1)}.$$
Tanaka's formula for local time gives
$$|B^{(0)}_t|=\int_{0}^{t}sgn(B^{(0)}_u)dB^{(0)}_u+L_t(B^{(0)})=B_t^{(1)}-\displaystyle{\min_{0\leq u \leq t}} B_u^{(1)},$$
where $L_t(B^{(0)})$ is the local time at $0$ of $B^{(0)}$ and so
$$B_t^{(1)}=\int_{0}^{t}\textrm{sgn}(B^{(0)}_s)dB^{(0)}_s.$$ 
The same reasoning shows that for all $h\geq 1$,
$$B_t^{(h+1)}=\int_{0}^{t}\textrm{sgn}(B^{(h)}_s)dB^{(h)}_s.$$ 
Thus the law of $Z$ is independent of the sequence $(m_n,n\in\N)$ and therefore
\begin{equation}\label{kil}
\left(S_n^0,S_n^1,\cdots,S_n^h,\cdots\right)\xrightarrow[\text{$n\rightarrow +\infty $}]{\text{law}}\bigg(B,B^1,\cdots,B^h,\cdots\bigg)\ \ \textrm{in}\ \mathbb E
\end{equation}
where $B$ is a Brownian motion.
\subsection{Convergence in probability.}\label{tawn}
Let $B$ a Brownian motion and recall the notations in Theorem \ref{g}. For each $n\geq 1$, define
$$U^n=\bigg(B,S^{n,0},B^1,S^{n,1},\cdots,B^h,S^{n,h},\cdots\bigg).$$
and let $\mathbb Q_n$ be the law of $U^n$. Since $\mathcal T^h(S^n)$ is a simple random walk for each $(h,n)$, a similar argument as in the proof of Lemma \ref{gh} shows that $\{\mathbb Q_n, n\geq 1\}$ is tight on $\mathbb E$. Fix a sequence $(m_n,n\in\N)$ such that $U^{m_n}\xrightarrow[\text{$n\rightarrow +\infty $}]{\text{law}}U$ in $\mathbb E$. Using (\ref{kil}), we see that there exist two Brownian motions $X$ and $Y$ such that
$$U=\bigg(X,Y,X^1,Y^1,\cdots,X^h,Y^h,\cdots\bigg).$$
It is easy to check that if $\varphi : \mathbb W\longrightarrow\R$ is bounded and uniformly continuous, then $\varPsi(f,g)=\varphi(f-g)$ defined for all $(f,g)\in \mathbb W^2$ is also bounded uniformly continuous which comes from 
$$d_U(f-f',g-g')=d_U(f-g,f'-g')\ \textrm{for all}\ f,f',g,g'\in\mathbb W.$$ 
Thus if $(F_n,G_n)$ converges in law to $(F,G)$ in $\mathbb W^2$, then $F_n-G_n$ converges in law to $F-G$ in $\mathbb W$. Applying this, we see that $B-S^{n,0}$ converges in law to $X-Y$. On the other hand, $B-S^{n,0}$ converges to $0$ (in $\mathbb W$) in probability (see \cite{MR0345224} page 39). Consequently $X=Y$ and 
$$U^n\xrightarrow[\text{$n\rightarrow +\infty $}]{\text{law}}\left(B,B,B^1,B^1,\cdots,B^h,B^h,\cdots\right)\ \ \textrm{in}\ \mathbb E.$$
In particular for each $h$, $S^{n,h}-B^h$ converges in law to $0$ as $n\rightarrow\infty$, that is $S^{n,h}$ converges to $B^h$ in probability as $n\rightarrow\infty$. 
Now the following equivalences are classical
\begin{itemize}
\item [(i)] $\lim_{n\rightarrow\infty}U^n=(B,B^1,\cdots)$ in probability in $\mathbb E$.
\item[(ii)] $\lim_{n\rightarrow\infty} E[d(U^n,U)\wedge 1]=0$.
\item[(iii)] For each $h$, $\lim_{n\rightarrow\infty} E[d_U(S^{n,h},B^h)\wedge 1]=0$.
\item[(iv)] For each $h$, $\lim_{n\rightarrow\infty}S^{n,h}=B^h$ in probability in $\mathbb W$.
\end{itemize}
Since we have proved (iv), Theorem \ref{g} holds.
\subsection{Proof of Corollary \ref{yeux}.}
(i) Let $S$ be a SRW and $X_0,X_1,\cdots,X_{p}$ be $p+1$ independent Brownian motions (not necessarily defined on the same probability space as $S$). Fix $$0\leq t^0_1\leq\cdots\leq t^0_{i_0},\ \  0\leq t^1_1\leq\cdots\leq t^1_{i_1},\cdots,0\leq t^p_1\leq\cdots\leq t^p_{i_p}.$$
By Corollary \ref{ko} (ii), for $n$ large enough (such that ${\lfloor nt^0_{i_0}\rfloor}+1\leq {\lfloor n\alpha^1_n\rfloor}$),  $\big(S^0_n(t^0_1),\cdots,S^0_n(t^0_{i_0})\big)$ which is $\sigma(S_j, j\leq {\lfloor nt^0_{i_0}}\rfloor+1)$-measurable, is independent of $\mathcal T^{{\lfloor n\alpha^1_n\rfloor}}(S)$. Thus
$(S^0_n(t^0_1),\cdots,S^0_n(t^0_{i_0}))$ is independent of $(S^{\lfloor n\alpha^1_n\rfloor}_n,\cdots,S^{\lfloor n\alpha^p_n\rfloor}_n)$ and similarly $\mathcal T^{{\lfloor n\alpha^2_n\rfloor}}(S)$ is independent of $\sigma(\mathcal T^{{\lfloor n\alpha^1_n\rfloor}}(S)_j, j\leq {\lfloor n\alpha^2_n\rfloor}-{\lfloor n\alpha^1_n\rfloor})$. Again, for $n$ large (such that ${\lfloor nt^1_{i_1}\rfloor}+1\leq {\lfloor n\alpha^2_n\rfloor}-{\lfloor n\alpha^1_n\rfloor}$), $\big(S^{\lfloor n\alpha^1_n\rfloor}_n(t^1_1),\cdots,S^{\lfloor n\alpha^1_n\rfloor}_n(t^1_{i_1})\big)$ is $\sigma(\mathcal T^{{\lfloor n\alpha^1_n\rfloor}}(S)_j, j\leq {\lfloor n\alpha^2_n\rfloor}-{\lfloor n\alpha^1_n\rfloor})$-measurable and therefore is independent of $\big(S^{\lfloor n\alpha^2_n\rfloor}_n,\cdots,S^{\lfloor n\alpha^p_n\rfloor}_n\big)$. By induction on $p$, for $n$ large enough, 
$$\big(S^0_n(t^0_1),\cdots,S^0_n(t^0_{i_0})\big),\ \ \big(S^{\lfloor n\alpha^1_n\rfloor}_n(t^1_1),\cdots,S^{\lfloor n\alpha^1_n\rfloor}_n(t^1_{i_1})\big),\cdots,\big(S^{\lfloor n\alpha^p_n\rfloor}_n(t^p_{1}),\cdots,S^{\lfloor n\alpha^p_n\rfloor}_n(t^p_{i_p})\big)$$ are independent and this yields the convergence in law of
$$\bigg(S^0_n(t^0_1),\cdots,S^0_n(t^0_{i_0}),S^{\lfloor n\alpha^1_n\rfloor}_n(t^1_1),\cdots,S^{\lfloor n\alpha^1_n\rfloor}_n(t^1_{i_1}),\cdots,S^{\lfloor n\alpha^p_n\rfloor}_n(t^p_{1}),\cdots,S^{\lfloor n\alpha^p_n\rfloor}_n(t^p_{i_p})\bigg)$$
to 
$$\bigg(X_0(t^0_1),\cdots,X_0(t^0_{i_0}),X_1(t^1_1),\cdots,X_1(t^1_{i_1}),\cdots,X_{p}(t^p_{1}),\cdots,X_{p}(t^p_{i_p})\bigg).$$
\noindent Thus the convergence of the finite dimensional marginals holds and the proof is completed.
\subsection{Proof of Proposition \ref{feda}.}
To prove part (i), we need the following lemma which may be found in \cite{MR773850} page 32 in more generality:
\begin{lemma}\label{hmoum}
If $(u_{k,n})_{k,n\in\N}$ is a nonegative and bounded doubly indexed sequence such that for all $k, \lim_{n\rightarrow\infty} u_{k,n}=0$, then there exists a nondecreasing sequence $(k_n)_n$ such that $\lim_{n\rightarrow\infty}k_n=+\infty$ and $\lim_{n\rightarrow\infty} u_{k_n,n}=0$.
\end{lemma}
\begin{proof}
By induction on $p$, we construct an increasing sequence $(n_p)_{p\in\N}$ such that $u_{p,n}<2^{-p}$ for all $n\geq n_p$. Now define 
$$\begin{array}{l}
k_n=\begin{cases}
n&\text{if}\  0\leq n\leq n_0\\
p\ &\text {if}\ n_p\leq n<n_{p+1}\ \textrm{for some}\ p\in\N.\\
\end{cases}
\end{array}$$
Clearly $n\longmapsto k_n$ is nondecreasing and $\lim_{n\rightarrow\infty}k_n=+\infty$. Moreover for all $p$ and $n\geq n_p$, we have $u_{k_n,n}<2^{-p}$. Thus for all $p$, $0\leq \limsup_{n\rightarrow\infty}u_{k_n,n}\leq 2^{-p}$ and since $p$ is arbitrary, the lemma is proved.
\end{proof}
\noindent The previous lemma applied to 
$$u_{k,n}=E[d_U(S^{n,k},B^k)\wedge 1],$$
guarantees the existence of a nondecreasing sequence $(\alpha^0_n)_n$ with values in $\N$ such that $\lim_{n\rightarrow\infty}\alpha^0_n=+\infty$ and 
\begin{equation}\label{ishake}
\lim_{n\rightarrow\infty}\big(S^{n,\alpha^0_n}-B^{\alpha^0_n}\big)=0\ \textrm{in probability in}\ \mathbb W.
\end{equation}
Now set
$$V^n=\bigg(B^{\alpha^0_n},S^{n,{\alpha^0_n}},B^{\alpha^0_n+1},S^{n,\alpha^0_n+1},\cdots,B^{\alpha^0_n+h},S^{n,\alpha^0_n+h},\cdots\bigg).$$
Using the same idea as in Section \ref{tawn} and the relation (\ref{ishake}), we prove that for all $j\in\N$,
\begin{equation}\label{ishak}
\lim_{n\rightarrow\infty}\big(S^{n,\alpha^0_n+j}-B^{\alpha^0_n+j}\big)=0\ \textrm{in probability in}\ \mathbb W.
\end{equation}
Equivalently : for all $j\in\N$,
$$\lim_{n\rightarrow\infty}u^0_{j,n}=0\ \ \textrm{where}\ \ u^0_{j,n}=E\big[d_U(S^{n,\alpha^0_n+j},B^{\alpha^0_n+j})\wedge 1\big].$$
By Lemma \ref{hmoum} again, there exists a nondecreasing sequence $(\beta^0_n)_n$ with values in $\N$ such that $\lim_{n\rightarrow\infty}\beta^0_n=+\infty$ and 
\begin{equation}\label{ishakee}
\lim_{n\rightarrow\infty}\big(S^{n,\alpha^0_n+\beta^0_n}-B^{\alpha^0_n+\beta^0_n}\big)=0\ \textrm{in probability in}\ \mathbb W.
\end{equation}
Define $\alpha^1_n=\alpha^0_n+\beta^0_n.$ Now using (\ref{ishakee}) and the same preceding idea, we construct $\alpha^2$ and all the $(\alpha^i)_i$ by the same way. Thus part (i) of Proposition \ref{feda} is proved.\\
\noindent To prove (ii), write 
$$\mathcal T(S^n)_{j+1}-\mathcal T(S^n)_{j}=\textrm{sgn}(S^n_{j+\frac{1}{2}})(S^n_{j+2}-S^n_{j+1}).$$
Thus for each, $k\geq 1$ and $i\geq 1$, 

$$\mathcal T^k(S^n)_i=\sum_{j=0}^{i-1}P^{n,k,j}(S^n_{j+k+1}-S^n_{j+k}),\ \ \textrm{with}\ \ P^{n,k,j}=\prod_{l=1}^{k}\textrm{sgn}(\mathcal T^{k-l}(S^n)_{j+l-\frac{1}{2}}).$$
Denote by $(\mathcal F_t)_{t\geq 0}$ the natural filtration of $B$, then $P^{n,k,j}$ is the product of $k$ random signs which are $\mathcal F_{T^n_{j+k}}$-measurable. This yields

$$\mathbb E\big[P^{n,k,j}(S^n_{j+k+1}-S^n_{j+k})|\mathcal F_{T^n_{j+k}}\big]=P^{n,k,j} \mathbb E\big[\sqrt{n}(B_{T^n_{j+k+1}}-B_{T^n_{j+k}})|\mathcal F_{T^n_{j+k}}\big]=0.$$
Consequently $E[\mathcal T^k(S^n)_i|\mathcal F_{T^n_{j+k}}]=0$ and a fortiori
$$\mathbb E\big[S^{n,k}(t)|\mathcal F_{T^n_{k}}\big]=0\ \ \textrm{for all}\ \ n, k\ \textrm{and}\ t.$$
\noindent Suppose there exists $t>0$ (which is fixed from now) such that $\lim_{n\rightarrow\infty}\big(S^{n,{h_n}}(t)-B^{{h_n}}_t\big)=0$ in probability; we will show that $\frac{h_n}{n}$ must tend to $0$. By Burkholder's inequality, we have
$$\mathbb E[|S_j|^p]\leq C_p \mathbb E[\big(S_1^2+(S_2-S_1)^2+\cdots+(S_j-S_{j-1})^2\big)^{\frac{p}{2}}]=C_p j^{\frac{p}{2}}.$$
\noindent Hence the $L^p$-norm of $S^{n,h_n}(t)$ is bounded uniformly in $n$ and the same is true for the $L^p$-norm of $B^{h_n}_t$ because $B^{h_n}$ is a Brownian motion. As a consequence, $S^{n,h_n}(t)-B^{h_n}_t$ tends to $0$ also in $L^p$-spaces.\\
From $\mathbb E\big[S^{n,h_n}(t)|\mathcal F_{T^n_{h_n}}\big]=0$ and using the $L^2$-continuity of conditional expectations, we get
$$\mathbb E\big[B^{h_n}_t|\mathcal F_{T^n_{h_n}}\big]\rightarrow 0\ \textrm{in}\ L^2.$$
Since $M_s=B^{h_n}_{t\wedge s}$ is a square-integrable $\mathcal F$-martingale; we have $E[B^{h_n}_t|\mathcal F_{T^n_{h_n}}]=B^{h_n}_{t\wedge T^n_{h_n}}$ and $B^{h_n}_{t\wedge T^n_{h_n}}$ must therefore tend to $0$ in $L^2$. So
$$0=\lim_{n\rightarrow\infty}\mathbb E\big[\big(B^{h_n}_{t\wedge T^n_{h_n}}\big)^2\big]=\lim_{n\rightarrow\infty}\mathbb E[{t\wedge T^n_{h_n}}].$$
This means that $T^n_{h_n}\rightarrow0$ in probability. Now recall the following
\begin{lemma} (see \cite{MR0345224} page 39.)
 The sequence of continuous-time processes $(\Delta^n)_n$ defined by
 $$\Delta^n(t)=\frac{k}{n}\ \textrm{for}\ t\in[T^n_k,T^n_{k+1}[$$
 converges uniformly on compacts in probability to the identity process $t$.
\end{lemma}
\noindent This Lemma implies that $\Delta^n(T^n_{h_n})\rightarrow0$ in probability. But $\Delta^n(T^n_{h_n})=\frac{h_n}{n}$, so that $\frac{h_n}{n}\rightarrow0.$

\section{Concluding remarks.}\label{ra}
We first notice that with a little more work, Theorem \ref{g}(ii) remains true when the Cs\'aki-Vincze transform is replaced with the Dubins-Smorodinsky, Fujita transform or any other \textquotedblleft reasonable \textquotedblright discrete version. Concerning Proposition \ref{feda}, it is clear that there is no contradiction between the two statements. In fact, by the proof of Lemma \ref{hmoum}, the sequences $(\alpha_i)_{i\in\N}$ are constructed such that $0\leq \alpha^0_n\leq n$ and $0\leq \alpha^i_n-\alpha^{i-1}_n\leq n$ for all $i$ and $n$. Let us now explain our interest in relation (\ref{em}). Suppose there exists $(h_n)_n$ with $\frac{h_n}{n}\rightarrow\infty$ and such that (\ref{em}) is satisfied. Then using the convergence of $S^{n,0}$ to $B$ and Corollary \ref{ko} (ii) applied to $S^n$, we can show that $(B,B^{h_n})$ converges in law to a $2$-dimensional Brownian motion. This is equivalent (see Proposition 17 in \cite{MR54301010}) to the ergodicity of the continuous L\'evy transformation on path space. Corollary \ref{yeux} asserts that this convergence holds in discrete time. However as proved before, such sequence $(h_n)_n$ does not exist and so the possible ergodicity of $T$ cannot be established by arguments involving assymptotics of $\mathcal T^n$. Thus the impression that a thorough study of good discrete versions could lead to a better understanding of the conjectured ergodicity of $T$ may be misleading.\\

\bibliographystyle{plain}
\bibliography{Bil4}

\end{document}